\documentclass[reqno,12pt]{amsart}
\newcommand{\heat}{(\frac{d}{dt}-\Delta)}
\title
[Curvature Decay Estimates of MCF in Higher Codimensions]
{Curvature Decay Estimates of Graphical Mean Curvature Flow in Higher Codimensions}
\author[Knut Smoczyk]{\sc Knut Smoczyk$^\ast$}
\address{
$^\ast$Leibniz Universit\"at Hannover,
Institut f\"ur Differentialgeometrie und Riemann Center for Geometry and Physics,
Welfengarten 1,
30167 Hannover,
Germany}
\email{smoczyk@math.uni-hannover.de}
\author[Mao-Pei Tsui]{\sc Mao-Pei Tsui$^{\ast\ast}$}
\address{
$^{\ast\ast}$University of Toledo,
Department of Mathematics and Statistics,
2801 W. Bancroft St,
Toledo, Ohio 43606-3390}
\email{mao-pei.tsui@utoledo.edu}
\author[Mu-Tao Wang]{\sc Mu-Tao Wang$^{\ast\ast\ast}$}
\address{
$^{\ast\ast\ast}$Columbia University,
Department of Mathematics,
2990 Broadway,
New York, NY 10027}
\email{mtwang@math.columbia.edu}
\thanks{The first author was supported by the DFG (German Research Foundation).
The second author was partially  supported  by a  Collaboration Grant for Mathematicians from the Simons Foundation \#239677.
The third author was partially supported by National Science
Foundation Grant DMS 1105483.}%

\date{November 28, 2014}

\usepackage{amscd,amsfonts,mathrsfs,amsthm,enumerate}
\usepackage{amssymb, amsmath}          
\usepackage[shortalphabetic]{amsrefs}    

\newtheorem{thm}{Theorem}

\newtheorem{pro}{Proposition}
\newtheorem{lem}{Lemma}

\newtheorem{rem}{Remark}

\numberwithin{cor}{section}
\numberwithin{pro}{section} \numberwithin{dfn}{section}
\numberwithin{lem}{section}
\numberwithin{rem}{section}\numberwithin{equation}{section}

\newcommand{\R}{\mathbb R}

\begin{document}

\maketitle
\begin{abstract}

We derive pointwise curvature estimates for graphical mean curvature flows in higher codimensions for a flat ambient space. To the best of our knowledge, this is the first such estimates without assuming smallness of first derivatives of the defining map.
An immediate application is a convergence theorem of the mean curvature flow of  the  graph of an
area decreasing map between flat Riemann surfaces.\end{abstract}

\section{Introduction}
Let $\Sigma_1$ and $\Sigma_2$ be two compact Riemannian manifolds
and $M = \Sigma_1 \times \Sigma_2$ be the product manifold. We
consider a smooth map $f:\Sigma_1 \rightarrow \Sigma_2$ and denote
the graph of $f$ by $\Sigma$; $\Sigma$ is a submanifold of $M$ by
the embedding $id\times f$.
 We study the
deformation of $f$ by the mean curvature flow. The idea is to deform $\Sigma$ along its mean curvature vector field in $M$ with the hope
that $\Sigma$ will remain a graph.
  This is the negative gradient flow of the
volume functional and a stationary point is a ``minimal map"
introduced by Schoen in \cite{sch}.


To describe previous
results, we recall the differential of $f$, $df$, at each point of
$\Sigma_1$ is a linear map between the tangent spaces. The
Riemannian structures enable us to define the adjoint of $df$.
Let $\{\lambda_i\}$ denote the eigenvalues of $\sqrt{ (df)^T df}$,
or the singular values of $df$, where $(df)^T$ is the adjoint of
$df$. Note that $\lambda_i$ is always nonnegative. We say $f$ is
an {\it area decreasing map} if $\lambda_i \lambda_j < 1$ for any
$i\not=j$ at each point. In particular, $f$ is area-decreasing if $df$ has rank one everywhere.

In \cite{tw}, see also \cite{wa4, wa5, wa6}, it was proved that the area decreasing condition is preserved along
the mean curvature flow and that the following global existence and
convergence theorem holds.

{\bf Theorem }{\rm (\cite{tw}, 2004).\,}{\it
Let $\Sigma_1$ and
$\Sigma_2$ be compact Riemannian manifolds of constant sectional curvatures
$k_1$ and $k_2$ respectively. Suppose $k_1\geq |k_2|$, $k_1+k_2
\geq 0$ and $dim(\Sigma_1) \ge 2$. If $f$ is a smooth area decreasing map from $\Sigma_1$ to $\Sigma_2$, the mean curvature flow of the graph of $f$
 remains the graph of an area decreasing map and  exists for all time.
 Moreover, if $k_1+k_2 > 0$  then it  converges smoothly to the graph of
  a constant map.}

This result has been generalized to allow more general curvature conditions \cite{lee, ss}. For example, the convergence part can be established when $k_1+k_2=0$ and $k_1\geq |k_2|>0$ in \cite{lee}.
An important ingredient of these proofs is to use the positivity of $k_1$ to show that the gradient of $f$ approaches zero as $t\rightarrow \infty$. In \cite{ss} the convergence follows, if
the sectional curvatures $\operatorname{sec}_{\Sigma_1},\operatorname{sec}_{\Sigma_2}$
of $\Sigma_1$ and $\Sigma_2$ are not necessarily constant and satisfy
$$\operatorname{sec}_{\Sigma_1}>-\sigma,\quad\operatorname{Ric}_{\Sigma_1}
\ge(n-1)\sigma\ge(n-1)\operatorname{sec}_{\Sigma_2}$$
for some positive constant $\sigma$, where $n=\dim(\Sigma_1)$. In this case the positivity of $\operatorname{Ric}_{\Sigma_1}$ is important to get the convergence.
However, in all cases mentioned above the convergence part in the case $k_1=k_2=0$ remains an open standing problem.

In general, the global existence and convergence of a mean curvature flow relies on the boundedness of the second fundamental form. In the above theorem, the boundedness of the second fundamental form is obtained by an indirect blow-up argument, see \cite{wa1, wa3, tw}. While the idea of the proof of convergence is to use the positivity of $k_1+k_2$ (or $k_1$ resp. $\operatorname{Ric}_{\Sigma_1}$) to show that
the gradient of $f$ is approaching zero, which in turn gives the boundedness of the second fundamental form when the flow exists for sufficiently long time. In \cite{ss2}
mean curvature estimates are shown in case of length decreasing maps ($\lambda_i<1$).
Other curvature estimates for higher co-dimensional graphical mean curvature flows have been obtained under various conditions \cite{cch,ccy}.
 However, to the best of our knowledge, there is no direct pointwise curvature estimate for higher codimensional mean curvature flow without
assuming smallness conditions on the Lipschitz norm of the first derivatives.  In this paper, we prove pointwise estimates assuming only a weaker
condition (area decreasing) on the gradient of the map. As a result,  the convergence of the flow can be established in dimension two when $k_1=k_2=0$.

\begin{thm}\label{thm1}
Let $(\Sigma_1,g_1)$ and $(\Sigma_2,g_2)$ be complete flat Riemann surfaces, $\Sigma_1$ being compact. Suppose $\Sigma\subset(\Sigma_1\times \Sigma_2,g_1\times g_2)$ is the graph of an
area decreasing map $f:\Sigma_1\to \Sigma_2$
and let $\Sigma_t$ denote its mean curvature flow with initial surface
$\Sigma_0=\Sigma$. Then $\Sigma_t$ remains the graph of an
area decreasing map $f_t$ along the mean curvature flow. The flow
exists smoothly for all time and $\Sigma_t$ converges smoothly to
a totally geodesic submanifold as $t\rightarrow \infty$.
Moreover, we have the following mean curvature decay estimate
\[t|H|^2 \leq  \frac{2}{\alpha}\] where $\alpha=inf_{\Sigma_0}\frac{2(1-\lambda_1^2\lambda_2^2)}{(1+\lambda_1^2)(1+\lambda_2^2)} >0$ and $\lambda_1$ and $\lambda_2$ are the singular values of $df$.
\end{thm}
\begin{rem}
Let $f:\Sigma_1\to\Sigma_2$ be an arbitrary smooth map between flat Riemann surfaces $(\Sigma_1,g_1)$, $(\Sigma_2,g_2)$
and suppose $\Sigma_1$ is compact. Then there exists a constant $c>0$ such that all singular values of $f$ satisfy $\lambda_i\lambda_j<c^2$. The map
$f:(\Sigma_1,g_1)\to(\Sigma_2,c^{-2}g_2)$ becomes area decreasing and we can apply Theorem
\ref{thm1} to this case since the new metric $\tilde g_2=c^{-2}g_2$ is still flat. Such a procedure only applies in the flat case and changes the Lipschitz norm of the map. Nevertheless, the geometry of the graph of $f$ is different. 
\end{rem}

As in \cite{sw1}, consider the symplectic structure $dx^1 \wedge dy^1+dx^2\wedge dy^2$ on $(T^2, \{x^i\}_{i=1,2})\times (T^2, \{y^j\}_{j=1,2})$
and suppose $\Sigma$ is Lagrangian with respect to this symplectic structure.
A stronger decay estimate on the second fundamental can be obtained in this case:

\begin{thm}\label{thm2}
Let $f:T^2\to T^2$ be an area decreasing map such that its graph $\Sigma$
is a Lagrangian submanifold in $T^2 \times T^2$ with respect to the above symplectic
structure, then the same conclusion as in Theorem \ref{thm1} holds and  \[t|A|^2 \leq C_\alpha,\] where $C_\alpha$ is a positive constant that only depends on $\alpha$.
\end{thm}

We first revisit the curvature estimates in codimension one by Ecker and Huisken \cite{eh1}. A direct generalization of their estimate only works in the higher codimensional case when the gradient of the defining function is small enough. However, we were able to reformulate their estimates in a different way that can be adapted to the higher codimensional case. It turns out in higher codimensions a more sophisticated approach has to be developed to accommodate the complexity of the normal bundle.

\noindent{\it Acknowledgements.} Part of this paper was completed
while the  authors were visiting Taida Institute of Mathematical
Sciences, National Center for Theoretical
Sciences, Taipei Office in National Taiwan University, Taipei, Taiwan and Riemann Center for Geometry and Physics in 
Leibniz Universit\"at Hannover. The authors wish to express their  gratitude for the excellent support they received during their stay.

\section{Ecker and Huisken's estimates in codimension one}
In this section, we slightly rewrite the estimate in \cite{eh1} so it can be adapted to the higher codimensional situation in later sections.
Consider the mean curvature flow of the graph of a function
$f:\R^n \rightarrow \R$ and let $v=\sqrt{1+|Df|^2}$.
Recall that the evolution equations of $v$ and $|A|^2$ are
\begin{equation*}
\begin{split}
 \heat v  = &  - |A|^2 v -2 \frac{|\nabla v|^2}{v} \text{ and }\\
 \heat |A|^2   = &  - 2|\nabla A|^2  +2 |A|^4.\\
\end{split}
\end{equation*}

We obtain the evolution equation of $\ln v^2$ as
\begin{equation}\label{eq_v^2}\heat \ln(v^2) =- 2|A|^2  - \frac{1}{2}|\nabla\ln (v^2)|^2.\end{equation}

Using $|\nabla A|^2 \geq |\nabla |A||^2 =\frac{|A|^2}{4}|\nabla \ln |A|^2|^2 $, we have
$\heat |A|^2   \leq  - \frac{|A|^2}{2}|\nabla \ln |A|^2|^2  +2 |A|^4.$
Taking $\ln$ of $|A|^2$, we obtain
\begin{equation}\label{eq_A^2}
\begin{split}
\heat \ln(|A|^2)  \leq 2|A|^2 + \frac{1}{2}|\nabla \ln |A|^2|^2.
\end{split}
\end{equation}
As in \cite{eh1}, equations \eqref{eq_v^2} and \eqref{eq_A^2} together imply a sup norm bound for $|A|^2 v^2$.

The following differential inequality for  $\ln(\delta t   |A|^2+\epsilon)$, which is similar to equation \eqref{eq_A^2}, gives a
decay estimate of $|A|^2$.
\begin{lem}\label{decay_codim_1}
 Given any $\epsilon >0$ and $\delta < 2\epsilon$. Then  \[ \heat \ln(\delta t   |A|^2+\epsilon) \leq  2|A|^2  + \frac{1}{2} |\nabla \ln(\delta t   |A|^2+\epsilon)|^2.\]
\end{lem}
\begin{proof}

Using
$\triangle \ln(\delta t|A|^2+\epsilon)
= \frac{\delta t\triangle |A|^2}{\delta t|A|^2+\epsilon}-\frac{\delta^2 t^2|\nabla  |A|^2|^2 }{(\delta t |A|^2+\epsilon)^2}$
, we  compute the evolution equation of $\ln(\delta t|A|^2+\epsilon)$:
\begin{eqnarray*}
&& \heat \ln(\delta t   |A|^2+\epsilon) \\
&=&  \frac{1}{\delta t   |A|^2+\epsilon} \heat (\delta t   |A|^2+\epsilon)  +
|\nabla \ln(\delta t   |A|^2+\epsilon)|^2 \\
&\leq&  \frac{1}{\delta t   |A|^2+\epsilon}\Big(\delta |A|^2 + \delta t   ( 2 |A|^4- \frac{|\nabla |A|^2|^2}{2|A|^2})\Big)+|\nabla \ln(\delta t   |A|^2+\epsilon)|^2 \\
&\leq &   \frac{\delta |A|^2+2[\delta t   |A|^2+\epsilon] |A|^2-2\epsilon |A|^2}{\delta t   |A|^2+\epsilon} -
\frac{1}{2}\frac{  \delta t |\nabla |A|^2|^2}{(\delta t   |A|^2+\epsilon)|A|^2} \\
&&+|\nabla \ln(\delta t   |A|^2+\epsilon)|^2\\
&\leq&   2|A|^2+ \frac{\delta |A|^2-2\epsilon |A|^2}{\delta t   |A|^2+\epsilon}  -
\frac{1}{2}\frac{  \delta t |\nabla |A|^2|^2}{(\delta t   |A|^2+\epsilon)|A|^2}
+|\nabla \ln(\delta t   |A|^2+\epsilon)|^2\\
&\leq & 2|A|^2 + \frac{1}{2} |\nabla \ln(\delta t   |A|^2+\epsilon)|^2.
\end{eqnarray*}
Here we use $\delta -2\epsilon< 0$ and
$$-\frac{1}{2}\frac{  \delta t |\nabla |A|^2|^2}{(\delta t   |A|^2+\epsilon)|A|^2}+
\frac{1}{2} |\nabla \ln(\delta t   |A|^2+\epsilon)|^2 \leq 0.$$
\end{proof}

\begin{thm}
$sup_{\Sigma_t}(t|A|^2) \leq v_0^2$ where $v_0= sup_{\Sigma_0} v > 0$.
\end{thm}
\begin{proof}
From the evolution equation of $v$, we have $sup_{\Sigma_t} v^2 \leq v_0^2$. Choosing $\varepsilon=1$ and $\delta=1$ in the previous Lemma and combining with equation \eqref{eq_v^2}, we derive
\[\begin{split}\heat \ln( (t|A|^2+1)v^2) &  \leq \frac{1}{2} |\nabla \ln( t   |A|^2+1)|^2-\frac{1}{2} |\nabla \ln(  v^2)|^2\\
&  \leq   \frac{1}{2}  \nabla \ln( \frac{ t|A|^2+1} {v^2}) \cdot \nabla \ln (( t|A|^2+1)v^2).\end{split}\]
The maximum principle implies
$$sup_{\Sigma_t} \left((t|A|^2+1)v^2\right) \leq sup_{\Sigma_0}v^2.$$
Therefore $t|A|^2  \leq (t|A|^2+1)v^2 \leq v_0^2$.
\end{proof}
\section{Estimates in higher codimensions}

Our  basic set-up here is a mean curvature flow
$F:\Sigma\times[0,T)\rightarrow  M$ of an $n$ dimensional
submanifold $\Sigma$ inside an $n+m$ dimensional flat Riemannian
manifold $M$. Given any  tensor on $M$, we may consider
the pull-back tensor by $F_t$ and consider the evolution equation
with respect to the time-dependent induced metric on
$F_t(\Sigma)=\Sigma_t$. For the purpose of applying the maximum
principle, it suffices to derive the equation at a space-time
point. We write all geometric quantities in terms of orthonormal
frames, keeping in mind all quantities are defined independent of
choices of frames. At any point $p\in \Sigma_t$, we choose any
orthonormal frames $\{e_i\}_{i=1,\cdots, n}$ for $T_p\Sigma_t$ and
$\{e_{\alpha}\}_{\alpha=n+1, \cdots, n+m}$ for the normal space $N_p\Sigma_t$. The
second fundamental form $h_{\alpha ij}$ is denoted by $h_{\alpha
ij} =\langle\nabla_{e_i}^M e_j, e_{\alpha}\rangle$ and the mean
curvature vector is denoted by $H_\alpha=\sum_ih_{\alpha ii}$. For
any $j,k$, we pretend
\[h_{n+i, jk}=0\] if $i>m$. Also we denote $|A|^2= \sum_{\alpha,i,j}h_{\alpha ij}^2$ and
$|H|^2= \sum_{\alpha} H_{\alpha}^2$.

First, we recall the evolution equations for $|H|^2$ and $|A|^2$.
The following proposition is taken from Corollary 3.8 and Corollary 3.9 from the survey paper \cite{sm2}.

\begin{pro} Suppose $M$ is a flat Riemannian manifold. 
For a mean curvature flow $F:\Sigma\times
[0, T)\rightarrow M$ of any dimension,
the quantities $|A|^2$ and $|H|^2$ satisfy the following
equations along the mean curvature flow:
\begin{equation}\label{|A|^2}
\begin{split}
& \frac{d}{dt}|A|^2
=  \Delta |A|^2 -2|\nabla^{\bot} A|^2\\
+ & 2\sum_{\alpha,\gamma, i,m}
(\sum_k h_{\alpha ik}h_{\gamma mk}
-h_{\alpha mk}h_{\gamma ik})^2
+2\sum_{i,j,m,k}(\sum_{\alpha} h_{\alpha ij}
h_{\alpha mk})^2
\end{split}
\end{equation}
and
\begin{equation}\label{|H|^2}
\begin{split}
\frac{d}{dt}|H|^2
&=\Delta |H|^2 -2|\nabla^{\bot} H|^2+ 2 \sum_{i,k}(\sum_{\alpha}H_{\alpha}h_{\alpha ik})^2.
\end{split}
\end{equation}

\end{pro}

Using Theorem 1 from \cites{LiLi92}, we have
$$2\sum_{\alpha,\gamma, i,m}
(\sum_k h_{\alpha ik}h_{\gamma mk}
-h_{\alpha mk}h_{\gamma ik})^2
+2\sum_{i,j,m,k}(\sum_{\alpha} h_{\alpha ij}
h_{\alpha mk})^2 \leq 3 |A|^4.$$ (This improves the prior bound of $4|A|^4$ used in \cite{wa1}).
Using
$$2 \sum_{i,k}(\sum_{\alpha}H_{\alpha}h_{\alpha ik})^2 \leq 2 |A|^2|H|^2,$$
we obtain the next lemma.

\begin{lem}\label{Aineq} Suppose $M$ is a flat Riemannian manifold. 
For a mean curvature flow $F:\Sigma\times
[0, T)\rightarrow M$ of any dimension,
we have the following differential inequalities for $|A|^2$ and $|H|^2$.
\begin{equation}\begin{split}
\heat |A|^2 &\leq -2|\nabla^\perp  A|^2 + 3 |A|^4,\\
\heat |H|^2 &\leq -2|\nabla^\perp  H|^2 + 2 |A|^2|H|^2.
\end{split}\end{equation}\end{lem}
We note that the term $3|A|^4$ in the first equation is different from the term $2|A|^4$ in the codimension one case. It cannot be improved unless the normal bundle is flat. This creates a major difficulty in attempting to generalize the codimension one estimate to the higher codimension case.

In the following, we derive differential inequalities for various geometric quantities which will be used for curvature decay estimates in \S 4.

\begin{lem}\label{Adecay} Suppose $M$ is a flat Riemannian manifold. 
For a mean curvature flow $F:\Sigma\times
[0, T)\rightarrow M$ of an $n$-dimensional submanifold $\Sigma$.
Given any $\epsilon >0$ and $0<\delta \leq \frac{2\epsilon}{n}$, we have the following differential inequalities
\begin{equation}\begin{split}\heat \ln (\delta t|H|^2 + \epsilon) &\leq 2|A|^2 + \frac{|\nabla \ln (\delta t|H|^2 + \epsilon)|^2}{2},\\
\heat \ln (\delta t|A|^2 + \epsilon)& \leq 3|A|^2 + \frac{|\nabla \ln (\delta t|A|^2 + \epsilon)|^2}{2}.\end{split}\end{equation}
\end{lem}

\begin{proof}

From Lemma \ref{Aineq}, we have
\begin{equation*}
\begin{split}\heat \ln(|H|^2) =  & \frac{1}{|H|^2}  \heat |H|^2  + |\nabla  \ln |H|^2|^2\\
\leq  & \frac{1}{|H|^2}\Big(2|A|^2|H|^2- 2|\nabla |H||^2\Big)+ |\nabla  \ln |H|^2|^2\\
\leq  & 2|A|^2  + \frac{1}{2} |\nabla \ln |H|^2|^2.
\end{split}
\end{equation*}
In the last step, we have used $ -\frac{2|\nabla |H||^2}{|H|^2}= -\frac{|\nabla  \ln |H|^2|^2}{2}$.

Using $\heat (\delta t|H|^2+\epsilon)=\delta |H|^2+ \delta t \heat |H|^2$ and Lemma \ref{Aineq}, we  compute the evolution equation of $\ln(\delta t|H|^2+\epsilon)$ to obtain
\begin{equation*}
\begin{split} & \heat \ln(\delta t|H|^2+\epsilon) \\
=  & \frac{1}{\delta t|H|^2+\epsilon} \heat (\delta t|H|^2+\epsilon)  +|\nabla \ln(\delta t |H|^2+\epsilon)|^2\\
\leq  & \frac{1}{\delta t|H|^2+\epsilon}\Big(\delta |H|^2 + \delta t( 2 |A|^2|H|^2- 2|\nabla |H||^2) \Big)+|\nabla \ln(\delta t |H|^2+\epsilon)|^2 \\
= &  \frac{\delta |H|^2+2(\delta t|H|^2+\epsilon) |A|^2-2\epsilon |A|^2}{\delta t|H|^2+\epsilon} -\frac{ 2 \delta t|\nabla |H||^2}{\delta t|H|^2+\epsilon} +|\nabla \ln(\delta t |H|^2+\epsilon)|^2\\
=  & 2|A|^2+ \frac{\delta |H|^2-2\epsilon |A|^2}{\delta t|H|^2+\epsilon}  -\frac{ \delta t|\nabla |H|^2|^2}{2(\delta t |H|^2+\epsilon)|H|^2} +|\nabla \ln(\delta t |H|^2+\epsilon)|^2.
\end{split}
\end{equation*}
In the last step, we have used $|\nabla |H||^2=\frac{|\nabla |H|^2|^2}{4|H|^2}$.
Since $|H|^2 \leq n |A|^2$ and
$-\frac{ \delta t|\nabla |H||^2}{2(\delta t |H|^2+\epsilon)|H|^2} +
\frac{1}{2}|\nabla \ln(\delta t |H|^2+\epsilon)|^2 \leq 0$
 , we can choose $\delta \leq \frac{2\epsilon}{n}$ and get
\begin{equation*} \heat \ln(\delta t|H|^2+\epsilon)
\leq  2|A|^2  + \frac{1}{2} |\nabla \ln(\delta t|H|^2+\epsilon)|^2.
\end{equation*}

The rest of the Lemma can be proved in a similar fashion to Lemma \ref{decay_codim_1}.
\end{proof}

To derive an a priori curvature estimate, we need to find the right geometric quantity to counteract the quadratic growth of the second fundamental forms in Lemma \ref{Adecay}. When $n=2$, we are able to find the right quantity to establish the a prioi mean curvature estimate.

In \cite{tw}, a parallel symmetric two tensor $S$ is introduced
to study the area decreasing map. We first recall some basic notations and definitions from Section 3 and Section 4 in \cite{tw}.

 When $M=\Sigma_1\times \Sigma_2$ is the product of $\Sigma_1$ and $\Sigma_2$, we denote the projections by
$\pi_1:M\rightarrow \Sigma_1$ and $\pi_2:M\rightarrow \Sigma_2$.
By abusing notations, we also denote the differentials by
$\pi_1:T_p M\rightarrow T_{\pi_1(p)}\Sigma_1$ and $\pi_1:T_p
M\rightarrow T_{\pi_2(p)}\Sigma_2$ at any  point $p\in M$.

When $\Sigma$ is the graph of $f:\Sigma_1\rightarrow \Sigma_2$,
the equation at each point can be written in terms of the singular
values of $df$ and special bases adapted to $df$. Denote the
singular values of $df$, or eigenvalues of $\sqrt{(df)^T df}$, by
$\{\lambda_i\}_{i=1,\cdots, n}$. Let $r$ denote the rank of $df$.
 We can rearrange them so that
$\lambda_i=0$ when $i$ is greater than $r$. By singular value
decomposition, there exist orthonormal bases $\{a_i\}_{i=1, \cdots,
n}$ for $T_{\pi_1(p)}\Sigma_1$ and $\{a_{\alpha}\}_{
\alpha=n+1,\cdots, n+m}$ for $T_{\pi_2(p)} \Sigma_2$ such that
\[df(a_i)=\lambda_i a_{n+i}\,\,\] for $i$ less than or equal to $r$
and $df(a_i)=0$ for $i$ greater than $r $. Moreover,
\begin{equation}\label{onf1_1} e_i=
\begin{cases} \frac{1}{\sqrt{1+\lambda_i^2}}(a_i+\lambda_i
a_{n+i})& \text{if}\,\, {1 \le i \le r   }\\
 a_i & \text{if}\,\, {r+1 \le i \le n   }
 \end{cases}
\end{equation}
 becomes an orthonormal basis for $T_p\Sigma$
 and
 \begin{equation}\label{onf1_3}
e_{n+p} =
\begin{cases}
\frac{1}{\sqrt{1+\lambda_p^2}}(a_{n+p} -\lambda_p
a_{p})& \text{if}\,\,1 \le p \le r\\
a_ {n+p}& \text{if}\,\,r+1 \le p \le m
\end{cases}
\end{equation}
becomes an orthonormal basis for $N_p\Sigma$.

The
tangent space of $M=\Sigma_1 \times \Sigma_2$ is identified with
$T\Sigma_1 \oplus T\Sigma_2$. Let $\pi_1$ and $\pi_2$ denote the
projection onto the first and second summand in the splitting. We
define the parallel symmetric two-tensor $S$  by
\begin{equation}\label{sym}
S(X,Y)=\langle \pi_1(X), \pi_1(Y)\rangle\ -\langle \pi_2(X),
\pi_2(Y)\rangle
\end{equation}
 for any
 $X, Y \in TM $.

Let $\Sigma$ be the graph of $f:\Sigma_1\rightarrow \Sigma_1
\times \Sigma_2$.  $S$ restricts to a symmetric two-tensor on
$\Sigma$ and we can represent $S$ in terms of the orthonormal
basis (\ref{onf1_1}).

Let $r$ denote the rank of $df$. By (\ref{onf1_1}), it is not hard
to check
\begin{equation}\label{pi1}
\begin{split}
\pi_1(e_i) & = \ \frac{a_i} {\sqrt{1+{\lambda_i}^2}} \; ,
 \pi_2(e_i)= \ \frac{\lambda_ia_{n+i}} {\sqrt{1+{\lambda_i}^2}}
\,\, \
\text{for}\,\,  1 \le i \le r \,\, , \\
\text{and} \ \pi_1(e_i) & = \ a_i \;,
 \pi_2(e_i)= \ 0
\,\, \text{for}\,\,  r+1 \le i \le n.
\end{split}
\end{equation}
Similarly, by (\ref{onf1_3}) we have
 \begin{equation}\label{pi1_1}
\begin{split}
\pi_1(e_{n+p}) &= \ \frac{-\lambda_p a_p}
{\sqrt{1+{\lambda_p}^2}}\; , \pi_2(e_{n+p}) = \ \frac{a_{n+p}}
{\sqrt{1+{\lambda_p}^2}}\,\,
\text{for}\,\,  1 \le p \le r\;,\,\, \\
\,\, \text{and}\,\, \pi_1(e_{n+p}) &= 0 \; , \pi_2(e_{n+p}) =
a_{n+p}  \,\, \text{for}\,\,r+1 \le p \le m\; .
\end{split}
\end{equation}
From the definition of $S$, we have
\begin{equation}\label{S_ij}
\begin{split}
 S(e_i, e_j)
 =  \frac{1- \lambda_i^2}{1+{\lambda_i}^2} \delta_{ij}\; .
\end{split}
\end{equation}
In particular, the eigenvalues of $S$ are
\begin{equation}\label{eigenvalues}
 \frac{1-{\lambda_i}^2 }{1+{\lambda_i}^2}, \,i= 1, \cdots, n.
\end{equation}

Now,  at each point we express $S$ in terms of the orthonormal
basis $\{{e_i}\}_{i = 1, \cdots, n}$ and  $\{{e_{\alpha}}\}_{\alpha
= n+1, \cdots, n+m}$. Let $I_{k \times k}$ denote a $k$ by $k$
identity matrix. Then $S$ can be written in the block form
\begin{equation}\label{block}
 S = \Big(
S(e_k,e_l) \Big)_{ 1 \le k,l \le {n+m}}=
 \left(
  \begin{matrix}
 B & 0 & D & 0 \\
0 & I_{n-r \times n-r} & 0 & 0 \\
D & 0 & -B & 0 \\
0 & 0 &0 &  - I_{m-r \times m-r}
\end{matrix}
\right)
\end{equation}
where $B$ and $D$ are  $r$ by $r$ matrices with $ B_{ij}=
S(e_i,e_j) = \frac{1-\lambda_i^2}{1+\lambda_i^2}\delta_{ij}$ and $
D_ {ij}= S(e_i,e_{n+j}) = \frac{-2\lambda_i}{1+\lambda_i^2}
\delta_{ij} $ for  $1 \le i, j \le r$.

Next we recall the evolution equation of parallel two-tensors from
\cite{sw1}.
 Given a parallel two-tensor $S$ on
$M$, we consider the evolution of $S$ restricted to $\Sigma_t$.
This is a family of time-dependent symmetric two tensors on
$\Sigma_t$.

\begin{pro}\label{evo_sym_l} Suppose $M$ is a flat Riemannian manifold. 
For a mean curvature flow $F:\Sigma\times
[0, T)\rightarrow M$ of an $n$-dimensional submanifold.
Let $S$ be a parallel two-tensor on $M$.   Then the pull-back of
$S$ to $\Sigma_t$ satisfies the following equation.
\begin{equation}\label{evo_symb}
\begin{split}
\heat S_{ij}&= - h_{\alpha i l}H_{\alpha}S_{lj}
- h_{\alpha j l}H_{\alpha}S_{li} \\
& + R_{k i k \alpha}S_{\alpha j} +  R_{k j k\alpha }S_{\alpha i}\\
& +h_{\alpha k l}h_{\alpha k i}S_{lj}
  +h_{\alpha k l}h_{\alpha k j}S_{li}
-2h_{\alpha k i} h_{\beta k j}S_{\alpha \beta}
\end{split}
\end{equation}
 where $\Delta $
is the rough Laplacian on two-tensors over $\Sigma_t$ and
$S_{\alpha i}=S(e_\alpha, e_i)$, $S_{\alpha \beta}=S(e_{\alpha},
e_{\beta})$, and $R_{k i k \alpha}=R(e_k, e_i, e_k, e_\alpha)$ is
the curvature of $M$.
\end{pro}

The  evolution equations (\ref{evo_symb}) of $S$ can be written in
terms of  evolving orthonormal frames as in Hamilton \cite{ha3}.
If the orthonormal frames
\begin{equation}\label{onf}
F=\{F_1, \cdots, F_a, \cdots,F_{n}\}
\end{equation}
are given in local coordinates by
\[
F_a = F_a^i \frac{\partial }{\partial x_i}\; .
\]
To keep them orthonormal, i.e. $ g_{ij}F_a^iF_b^j = \delta_{ab}$,
we evolve $F$ by the formula
\[
\frac{\partial }{\partial t} F_a^i = g^{ij}g^{\alpha\beta}
h_{\alpha j l}H_{\beta}F_a^l \, .
\]

Let $S_{ab} = S_{ij}F_a^iF_b^j$ be the components of $S$ in $F$.
Then $S_{ab}$ satisfies the following equation
\begin{equation}\label{onframes}
\begin{split}
\heat S_{ab}&=  \
 R_{c a c \alpha}S_{\alpha b} +  R_{c b c \alpha}S_{\alpha a}\\
& \ +h_{\alpha c d}h_{\alpha c a}S_{db}
  +h_{\alpha c d}h_{\alpha c b}S_{da} \\
& \ -2h_{\alpha c a} h_{\beta c b}S_{\alpha \beta} \, .
\end{split}
\end{equation}

In the following, we will compute the evolution of $Tr(S)=\sum_{i,j}g^{ij}S_{ij}$ in the case $\Sigma_1, \Sigma_2$ are flat Riemann surfaces, i.e. where the curvature tensor $R=0$. By equation \eqref{S_ij},
$Tr(S)=\frac{1-\lambda_1^2}{1+\lambda_1^2}+\frac{1-\lambda_2^2}{1+\lambda_2^2}
=\frac{2(1-\lambda_1^2\lambda_2^2)} {(1+\lambda_1^2)(1+\lambda_2^2)}$. The next proposition gives a new proof that the area
decreasing condition is preserved along the mean curvature. In addition, the equation satisfied by $\ln Tr(S)$ plays a critical
role in next section's curvature estimates.

\begin{pro}\label{etrS} Suppose $M=\Sigma_1\times \Sigma_2$ where $\Sigma_1$ and $\Sigma_2$ are complete flat Riemann surfaces. Suppose $\Sigma$ is the graph of a map $f:\Sigma_1\to \Sigma_2$
and let $\Sigma_t$ denote its mean curvature flow with initial surface
$\Sigma_0=\Sigma$.  Let $S$ be given by \eqref{sym}, then $Tr(S)$ on $\Sigma_t$ satisfies the following equation
\begin{equation}\label{lnTrS}
\begin{split}
& \heat  \ln Tr(S) \\
 = & 2|A|^2  +\frac{ |\nabla \ln(TrS)|^2}{2}+\frac{ 2\sum_{c=1}^2(T_{11}h_{4c2}+T_{22}h_{3c1})^2  }{Tr(S)^2},
\end{split}
\end{equation}
\end{pro}
where $T_{11}=\frac{2 \lambda_1}{(1+\lambda_1^2)}$ and  $T_{22}=\frac{2 \lambda_2}{(1+\lambda_2^2)}$.

\begin{proof}

Using the evolution equation of $S$ (with respect to an orthonormal frame) in equation (\ref{onframes})
and $S(e_{2+i},e_{2+j})=-S_{ij}$, we derive
\[
\begin{split}
& \heat Tr(S)\\
= &  \sum_{a,b}  \delta^{ab} (\sum_{\alpha, c, d}h_{\alpha c d}h_{\alpha c a}S_{db}
  +h_{\alpha c d}h_{\alpha c b}S_{da}  \ -2h_{\alpha c a} h_{\beta c b}S_{\alpha \beta})\\
= & \sum_{a}  (\sum_{\alpha, c}2h_{\alpha c a}^2S_{aa} -2h_{\alpha c a}^2S_{\alpha \alpha})\\
= & \sum_{a}(\sum_{p, c}2h_{2+p\,\, c a}^2S_{aa} +2h_{2+p \,\,c a}^2S_{pp})\\
=& \sum_{p, c}\left[2h_{2+p\,\, c 1}^2S_{11} +2h_{2+p\,\, c 1}^2S_{pp}+ 2h_{2+p\,\, c 2}^2S_{22} +2h_{2+p\,\, c 2}^2S_{pp}\right]\\
= &  \sum_c \left[(2h_{4c1}^2+2h_{3c2}^2)(S_{11}+S_{22})+ 4h_{3c1}^2 S_{11} +4h_{4c2}^2 S_{22}\right]\\
\end{split}
\]
Using $|A|^2= \sum_{c=1}^2 (h_{4c1}^2+h_{3c2}^2+h_{3c1}^2+h_{4c2}^2)$, we obtain
\begin{equation}\label{Tr(S)}
\heat Tr(S)= 2|A|^2 Tr(S)+ 2(S_{11}-S_{22})\sum_{c=1}^2(h_{3c1}^2-h_{4c2}^2).
\end{equation}

We claim the following relation holds:
\begin{eqnarray}
 && 4Tr(S) (S_{11}-S_{22}) \sum_{c=1}^2(h_{3c1}^2-h_{4c2}^2)+|\nabla Tr(S)|^2 \nonumber\\
&&=4 \sum_{c=1}^2 (T_{11}h_{4c2}+ T_{22}h_{3c1})^2.\label{relation}
\end{eqnarray}
Equation \eqref{lnTrS} follows from equations \eqref{Tr(S)} and \eqref{relation}.

In the rest of the proof, we verify equation \eqref{relation}. The covariant derivative of the restriction of
$S$ on $\Sigma$ can be computed by
\[
\begin{split}
& (\nabla_{e_k} S)(e_i,e_j)\\
= \quad & e_k(S(e_i,e_j))- S(\nabla_{e_k} e_i,e_j)-S(e_i,\nabla_{e_k} e_j) \\
= \quad & S(\nabla_{e_k}^{M}e_i -\nabla_{e_k}e_i,e_j)-
S(e_i,\nabla_{e_k}^{M}e_j-\nabla_{e_k} e_j)\\
= \quad & h_{\alpha k i}S_{\alpha j} + h_{\beta k j}S_{\beta i}.
\end{split}
\]

Since $S_{2+p\,\, l} =
 - \frac{2\lambda_p \delta_{pl}}{(1+\lambda_p^2)}$, we derive
\begin{equation}\begin{split}\label{dS}
S_{ij,k }= & - \frac{2h_{2+p\,\, k i}\lambda_p \delta_{pj}}{(1+\lambda_p^2)}
 - \frac{2h_{2+p\,\, k j}\lambda_p \delta_{pi}}{(1+\lambda_p^2)}\\
 = & - h_{2+p\,\, k i}T_{pp} \delta_{pj}
 - h_{2+p\,\, k j}T_{pp}\delta_{pi}\\
 = & - h_{2+j \,\,k i}T_{jj}
 - h_{2+i\,\, k j}T_{ii}.
\end{split}
\end{equation}

In particular, $
\nabla_k (S_{ii})
 =-2T_{ii}h_{2+i\,\,ki}$ and \[
|\nabla Tr(S) |^2
 =4\sum_{k=1}^2(T_{11}^2h_{3k1}^2+2T_{11}T_{22}h_{3k1}h_{4k2}+T_{22}^2h_{4k2}^2) .\]
We compute
\begin{equation*}
\begin{split}
 & 4(S_{11}^2-S_{22}^2) \sum_{k=1}^2(h_{3k1}^2-h_{4k2}^2)+|\nabla Tr(S)|^2 \\
&= 4(S_{11}^2-S_{22}^2) \sum_{k=1}^2(h_{3k1}^2-h_{4k2}^2)\\
& \quad+4\sum_{k=1}^2(T_{11}^2h_{3k1}^2+2T_{11}T_{22}h_{3k1}h_{4k2}+T_{22}^2h_{4k2}^2)\\
&= 4 \sum_{k=1}^2 (T_{11}h_{4k2}+ T_{22}h_{3k1})^2,\\
\end{split}
\end{equation*}
where we use the fact that $S_{11}^2+ T_{11}^2=S_{22}^2+ T_{22}^2=1$ to complete the square in the last equality. This verifies
\eqref{relation}.
\end{proof}


\section{Proof of Theorems}
\subsection{Proof of Theorem \ref{thm1}}
Since $Tr(S)=\frac{2(1-\lambda_1^2\lambda_2^2)}{(1+\lambda_1^2)(1+\lambda_2^2)}$, by Proposition \ref{etrS} and the maximum principle, we deduce that the area decreasing condition is preserved by the mean curvature flow and
$inf_{\Sigma_t}{Tr(S)} \geq \alpha$, where $\alpha=  inf_{\Sigma_0}{Tr(S)}$.
On any $\Sigma_t$, $t>0$,  $Tr(S) \geq \alpha > 0$ and
$(1+\lambda_1^2)(1+\lambda_2^2) < \frac{2}{\alpha}$. Thus $\Sigma_t$ remains as the graph of an area decreasing map.

Next we derive the mean curvature decay estimate.
 Combining the first equation in Lemma \ref{Adecay} (with $\delta=\epsilon=1$) and Proposition \ref{etrS}, we obtain
\[\heat \ln \frac{(t|H|^2+1)}{Tr(S)} \leq
 \frac{1}{2} \nabla \ln \frac{(t|H|^2+1)}{Tr(S)} \cdot \nabla  \ln \left[(t|H|^2+1) Tr(S)\right].\]
 By the maximum principle, we have
$$sup_{\Sigma_t} \ln \frac{(t|H|^2+1)}{Tr(S)} \leq sup_{\Sigma_0} \ln  \frac{1}{Tr(S)}.$$
 Note that $f_t$ remains area decreasing and
$$Tr(S)=\frac{2(1-\lambda_1^2\lambda_2^2)}{(1+\lambda_1^2)(1+\lambda_2^2)}<  2.$$
 This implies that $t|H|^2 \leq Tr(S) sup_{\Sigma_0} \frac{1}{Tr(S)} \leq sup_{\Sigma_0} \frac{2}{Tr(S)}=\frac{2}{\alpha}$. If $\Sigma_2$ is compact we already have longtime existence of the
flow by the methods in \cite{wa1, wa3, tw}. In case $\Sigma_2$ is complete and non-compact we can
use the mean curvature estimate to obtain a $C^0$-estimate on finite time intervals and then
we may proceed as in \cite{ss2} to get longtime existence.

Now we can prove the   $C^\infty$ convergence using the mean curvature decay estimate. By the Gauss formula,
$\int_{\Sigma_t}{|A|^2} d\mu_t=\int_{\Sigma_t}|H|^2 d\mu_t \rightarrow 0 $. Therefore $\int_{\Sigma_t}|A|^2 d\mu_t$ is sufficiently small when $t$ is large enough, the
$\epsilon$ regularity theorem in \cite{il} (see also \cite{ec})
implies $\sup_{\Sigma_t} |A|^2 $ is uniformly bounded.
The general convergence theorem of Simon
\cite{si} implies  $C^\infty$ convergence of $\Sigma_t$ to a minimal submanifold $\Sigma_{\infty}$, which is totally geodesic by the Gauss formula again.
\subsection{Proof of Theorem 2}

Now we prove a decay estimate for  the second fundamental form in the case when $\Sigma$ is a Lagrangian submanifold.
It is well known that $\Sigma_t$ remains as a Lagrangian submanifold in $T^2\times T^2$ from \cite{sm1}.
Using the first equation in Lemma \ref{Aineq} and Proposition \ref{etrS}, we derive
\begin{equation*}
\begin{split}
& \heat \ln( \frac{|A|^2}{Tr(S)^2 }) \\
\le &  \frac{ |\nabla \ln (|A|^2)|^2}{2}  - | \nabla \ln Tr(S)|^2-|A|^2-   \frac{4}{Tr(S)^2}
\sum_{c=1}^2(T_{11}h_{4c2}+T_{22}h_{3c1})^2.\\
=  &\frac{ 1}{2} \nabla \ln(\frac{|A|^2}{Tr(S)^2}) \cdot \nabla \ln (|A|^2 Tr(S)^2)-|A|^2\\
+&  | \nabla \ln Tr(S)|^2-   \frac{4}{Tr(S)^2}
\sum_{c=1}^2(T_{11}h_{4c2}+T_{22}h_{3c1})^2. \\
\end{split}
\end{equation*}

We estimate the last two terms on the right hand side. From equation \eqref{relation}, we obtain
\begin{equation*}
\begin{split}
&
| \nabla \ln Tr(S)|^2
- \frac{4}{Tr(S)^2}\sum_{c=1}^2 (T_{11}h_{4c2}+ T_{22}h_{3c1})^2 \\
&=  - \frac{4(S_{11}-S_{22})}{Tr(S)}\sum_{c=1}^2(h_{3c1}^2-h_{4c2}^2).
\end{split}
\end{equation*}

Let $J$ be the standard almost complex structure on $(T^2, \{x^i\}_{i=1,2})\times (T^2, \{y^j\}_{j=1,2})$ that maps from the tangent space
of the first component to the tangent space of the second one, and vice versa. Suppose $f$ is the
defining map of a Lagrangian surface $\Sigma$ in $(T^2, \{x^i\}_{i=1,2})\times (T^2, \{y^j\}_{j=1,2})$ with respect to $dx^1\wedge
dy^1+dx^2\wedge dy^2$. By the Lagrangian condition, $Jdf $ is a self-adjoint map on the tangent space of $(T^2, \{x^i\}_{i=1,2})$. Therefore,
there exists an orthonormal basis $\{a_i\}_{i=1,2}$ such that
\[df(a_i)=\lambda_i J(a_i), i=1, 2.\]

We can then choose $a_{2+i}=J (a_i)$  in equations \eqref{onf1_1} and \eqref{onf1_3} such that
\[e_i=\frac{1}{\sqrt{1+\lambda_i^2}}(a_i+\lambda_i J(a_{i}) )\] and
\[e_{2+i}=\frac{1}{\sqrt{1+\lambda_i^2}}(J(a_{i}) -\lambda_i
a_{i}), i=1, 2.\] From these expressions, it is easy to check that $e_3=J(e_1)$ and $e_4=J(e_2)$. Since $\Sigma_t$ remains Lagrangian, such orthonormal
frames can be picked at any point on $\Sigma_t$.  Because $J$ is parallel, \[\langle \nabla^M_{e_c} e_2, e_3\rangle= \langle \nabla^M_{e_c} e_1, e_4\rangle\] and we have $h_{3c2}^2=h_{4c1}^2$ for $c=1,2$. Therefore,
\[\Big| \sum_{c=1}^2(h_{3c1}^2-h_{4c2}^2)  \Big|\leq |h_{311}-h_{322}||H_3|+ |h_{411}-h_{422}||H_4|\leq 2\sqrt{2}|A||H|\] and
 \begin{equation*}|| \nabla \ln Tr(S)|^2
- \frac{4}{Tr(S)^2} \sum_{c=1}^2 (T_{11}h_{4c2}+ T_{22}h_{3c1})^2 |\leq C|A||H|,
\end{equation*}
where $C$ depends only on  ${\alpha}=inf_{\Sigma_0} \frac{2(1-\lambda_1^2\lambda_2^2)}{(1+\lambda_1^2)(1+\lambda_2^2)} > 0$ on the initial surface.
For example, $C=\frac{16\sqrt{2}}{{\alpha}}$ suffices.
To this end,
\begin{eqnarray*}
 \heat \ln( \frac{|A|^2}{Tr(S)^2 })
&\le&  \frac{ 1}{2} \nabla \ln(\frac{|A|^2}{Tr(S)^2}) \cdot \nabla \ln (|A|^2 Tr(S)^2)\\
&&  -|A|^2+C|A||H|.
\end{eqnarray*}
Note that $ -|A|^2 \leq - {\alpha}^2\frac{|A|^2}{Tr(S)^2 } $ and $|H|^2\leq \frac{2}{t{\alpha}}$ from Theorem 1, we derive
\[-|A|^2+C|A||H| \leq -\frac{1}{2}|A|^2+\frac{C^2}{2}|H|^2\le -\frac{{\alpha}^2|A|^2}{2Tr(S)^2}+\frac{C^2}{{\alpha}t}.\]

Therefore, the quantity $\frac{|A|^2}{Tr(S)^2 }$ satisfies the following differential inequality:
\begin{equation*}
\begin{split}
& \heat ( \frac{|A|^2}{Tr(S)^2 }) \\
\le  & \frac{1}{2} \nabla (\frac{|A|^2}{Tr(S)^2}) \cdot \nabla \ln (|A|^2 Tr(S)^2)  +(-\frac{{\alpha}^2|A|^2}{2Tr(S)^2}+\frac{C^2}{{\alpha}t}) \frac{|A|^2}{Tr(S)^2 }.\\
\end{split}
\end{equation*}

Consider the ODE
$u' = (-\frac{{\alpha}^2}{2}u+\frac{C^2}{{\alpha}t})u$. Note that $w=\frac{2(1+\frac{C^2}{{\alpha}})}{{\alpha}^2t}$ is a solution to
$u'=(-\frac{{\alpha}^2}{2}u+\frac{C^2}{{\alpha}t})u$ and $\lim_{t\to 0^+}w(t)=\infty$.
Thus we have $sup_{\Sigma_t}(\frac{|A|^2}{Tr(S)^2 })\leq \frac{2(1+\frac{C^2}{{\alpha}})}{{\alpha}^2t}$ and
$|A|^2 \leq\frac{8(1+\frac{C^2}{{\alpha}})}{{\alpha}^2t}$.


\begin{bibdiv}
\begin{biblist}


\bib{cch}{article} {
    AUTHOR = {Chau, Albert},
   AUTHOR = {Chen, Jingyi},
   AUTHOR = {He, Weiyong},
     TITLE = {Lagrangian mean curvature flow for entire {L}ipschitz graphs},
   JOURNAL = {Calc. Var. Partial Differential Equations},
  FJOURNAL = {Calculus of Variations and Partial Differential Equations},
    VOLUME = {44},
      YEAR = {2012},
    NUMBER = {1-2},
     PAGES = {199--220},
      ISSN = {0944-2669},
   MRCLASS = {53C44 (35K55)},
  MRNUMBER = {2898776},
MRREVIEWER = {Robert Haslhofer},
       DOI = {10.1007/s00526-011-0431-x},
       URL = {http://dx.doi.org/10.1007/s00526-011-0431-x},
}

\bib{ccy}{article} {
 AUTHOR = {Chau, Albert},
   AUTHOR = {Chen, Jingyi},
   AUTHOR = {Yuan, Yu},
     TITLE = {Lagrangian mean curvature flow for entire {L}ipschitz graphs
              {II}},
   JOURNAL = {Math. Ann.},
  FJOURNAL = {Mathematische Annalen},
    VOLUME = {357},
      YEAR = {2013},
    NUMBER = {1},
     PAGES = {165--183},
      ISSN = {0025-5831},
   MRCLASS = {53C44 (35A01 35B65 35K15 35K55)},
  MRNUMBER = {3084345},
       DOI = {10.1007/s00208-013-0897-2},
       URL = {http://dx.doi.org/10.1007/s00208-013-0897-2},
}


\bib{ec}{article} {
    AUTHOR = {Ecker, Klaus},
     TITLE = {Regularity theory for mean curvature flow},
   JOURNAL = {Progress in Nonlinear Differential Equations and their
Applications, 57. Birkh\"auser Boston, Inc., Boston, MA, 2004.},
}

\bib{eh1}{article} {
    AUTHOR = {Ecker, Klaus },
    AUTHOR = {Huisken, Gerhard},
     TITLE = {Mean curvature evolution of entire graphs},
   JOURNAL = {Ann. of Math. (2)},
  FJOURNAL = {Annals of Mathematics. Second Series},
    VOLUME = {130},
      YEAR = {1989},
    NUMBER = {3},
     PAGES = {453--471},
      ISSN = {0003-486X},
     CODEN = {ANMAAH},
   MRCLASS = {53A10 (53C45)},
  MRNUMBER = {1025164 (91c:53006)},
MRREVIEWER = {S. Walter Wei},
       DOI = {10.2307/1971452},
       URL = {http://dx.doi.org/10.2307/1971452},
}
	



\bib{eh2}{article} {
     AUTHOR = {Ecker, Klaus },
    AUTHOR = {Huisken, Gerhard},
     TITLE = {Interior estimates for hypersurfaces moving by mean curvature},
   JOURNAL = {Invent. Math.},
  FJOURNAL = {Inventiones Mathematicae},
    VOLUME = {105},
      YEAR = {1991},
    NUMBER = {3},
     PAGES = {547--569},
      ISSN = {0020-9910},
     CODEN = {INVMBH},
   MRCLASS = {53A10 (35K55 58E12)},
  MRNUMBER = {1117150 (92i:53010)},
MRREVIEWER = {Friedrich Sauvigny},
       DOI = {10.1007/BF01232278},
       URL = {http://dx.doi.org/10.1007/BF01232278},
}


\bib{ha3}{article}{
author={Hamilton, Richard S.},
title={Harnack estimate for the mean curvature
flow},
journal={J. Differential Geom.},
    volume={41},
    date={1995},
    number={1},
    pages={215--226},
  }









\bib{il}
{article}{
   author={Tom Ilmanen},
   title={Singularities of mean curvature flow of surfaces},
   journal={preprint},
   date={1997}}
\bib{lee}  {article} {
    AUTHOR = {Lee, Kuo-Wei},
    AUTHOR = {Lee, Yng-Ing},
     TITLE = {Mean curvature flow of the graphs of maps between compact
              manifolds},
   JOURNAL = {Trans. Amer. Math. Soc.},
  FJOURNAL = {Transactions of the American Mathematical Society},
    VOLUME = {363},
      YEAR = {2011},
    NUMBER = {11},
     PAGES = {5745--5759},
      ISSN = {0002-9947},
     CODEN = {TAMTAM},
   MRCLASS = {53C44},
  MRNUMBER = {2817407 (2012h:53154)},
MRREVIEWER = {Oliver C. Schn{\"u}rer},
       DOI = {10.1090/S0002-9947-2011-05204-9},
       URL = {http://dx.doi.org/10.1090/S0002-9947-2011-05204-9},
}


\bib{LiLi92}{article}{
   author={Li, An-Min},
   author={Li, Jimin},
   title={An intrinsic rigidity theorem for minimal submanifolds in a
   sphere},
   journal={Arch. Math. (Basel)},
   volume={58},
   date={1992},
   number={6},
   pages={582--594},
}



\bib{sch}{article}{
   author={Richard Schoen},
title={The role of harmonic mappings in rigidity and
 deformation problems},
journal={Complex geometry
 (Osaka, 1990), 179--200, Lecture Notes in
 Pure and Appl. Math., 143, Dekker, New York,
 1993.}}



\bib{si}{article}{
   author= {Leon Simon},
     title={Asymptotics for
 a class of nonlinear evolution equations,
 with applications to geometric problems.}
journal={Ann. of Math. (2) 118 (1983), no. 3, 525--571.}
}



\bib{ss}{article}{
   author={Savas-Halilaj, Andreas},
   author={Smoczyk, Knut},
   title={Homotopy of area decreasing maps by mean curvature flow},
   journal={Adv. Math.},
   volume={255},
   date={2014},
   pages={455--473},
   issn={0001-8708},
   review={\MR{3167489}},
   doi={10.1016/j.aim.2014.01.014},
}

\bib{ss2}{article}{ 
author={Savas-Halilaj, Andreas},
author={Smoczyk, Knut},
title={Evolution of contractions by mean curvature flow},
journal={Mathematische Annalen},
date={2014},
pages={1-16},
issn={0025-5831},
doi={10.1007/s00208-014-1090-y},
url={http://dx.doi.org/10.1007/s00208-014-1090-y},
publisher={Springer Berlin Heidelberg},
}



\bib{sm1}{article}{
  author={Smoczyk, Knut},
   title={A canonical way to deform a Lagrangian submanifold},
   journal={arXiv: dg-ga/9605005},
   date={1996},
}

\bib{sm2}{article}{
   author={Smoczyk, Knut},
   title={Mean curvature flow in higher codimension-Introduction and survey},
 BOOKTITLE ={Global Differential Geometry},
  journal={Springer Proc. Math.,17, Part II},
   PAGES = {231--274},
 PUBLISHER = {Springer},
   ADDRESS = {Heidelberg},
      YEAR = {2012},
}

\bib{sm3}{article}{
    AUTHOR = {Smoczyk, Knut},
     TITLE = {Long time existence of the {L}agrangian mean curvature flow},
   JOURNAL = {Calc. Var. Partial Differential Equations},
  FJOURNAL = {Calculus of Variations and Partial Differential Equations},
    VOLUME = {20},
      YEAR = {2004},
    NUMBER = {1},
     PAGES = {25--46},
      ISSN = {0944-2669},
   MRCLASS = {53C44},
  MRNUMBER = {2047144 (2004m:53119)},
MRREVIEWER = {Henri Anciaux},
       DOI = {10.1007/s00526-003-0226-9},
       URL = {http://dx.doi.org/10.1007/s00526-003-0226-9},
}



\bib{sw1}{article}{
   author={Smoczyk, Knut},
   author={Wang, Mu-Tao},
   title={Mean curvature flows of Lagrangian submanifolds with convex
   potentials},
   journal={J. Differential Geom.},
   volume={62},
   date={2002},
   number={2},
   pages={243--257},
   issn={0022-040X},
}


\bib{tw}{article}{
    AUTHOR = {Tsui, Mao-Pei}
    author={Wang, Mu-Tao},
     TITLE = {Mean curvature flows and isotopy of maps between spheres},
   JOURNAL = {Comm. Pure Appl. Math.},
  FJOURNAL = {Communications on Pure and Applied Mathematics},
    VOLUME = {57},
      YEAR = {2004},
    NUMBER = {8},
     PAGES = {1110--1126},
      ISSN = {0010-3640},
     CODEN = {CPAMA},
   MRCLASS = {53C44},
  MRNUMBER = {2053760 (2005b:53110)},
MRREVIEWER = {Henri Anciaux},
       DOI = {10.1002/cpa.20022},
       URL = {http://dx.doi.org/10.1002/cpa.20022},
}


%


\bib{wa1}{article}{
  author={Wang, Mu-Tao},
   title={Mean curvature flow of surfaces in Einstein four-manifolds},
   journal={J. Differential Geom.},
   volume={57},
   date={2001},
   number={2},
   pages={301--338},
   issn={0022-040X},
   review={\MR{1879229 (2003j:53108)}},
}

\bib{wa2}{article}{
AUTHOR = {Wang, Mu-Tao},
     TITLE = {Deforming area preserving diffeomorphism of surfaces by mean
              curvature flow},
   JOURNAL = {Math. Res. Lett.},
  FJOURNAL = {Mathematical Research Letters},
    VOLUME = {8},
      YEAR = {2001},
    NUMBER = {5-6},
     PAGES = {651--661},
      ISSN = {1073-2780},
   MRCLASS = {53C44},
  MRNUMBER = {1879809 (2003f:53122)},
}

\bib{wa3}{article}{
  author={Wang, Mu-Tao},
    title={Long-time existence and convergence of graphic mean curvature flow
    in arbitrary codimension},
    journal={Invent. Math.},
    volume={148},
    date={2002},
    number={3},
    pages={525--543},
    issn={0020-9910},
     review={\MR{1908059 (2003b:53073)}},
 }		
 
\bib{wa4}{article}{
  author={Wang, Mu-Tao},
    title={Gauss maps of the mean curvature flow},
    journal={Math. Res. Lett},
    volume={10},
    date={2003},
    number={2-3}
    pages={287-299},
     review={\MR{1981905 (2004m:53121)}},
 }	

\bib{wa5}{article}{
  author={Wang, Mu-Tao},
    title={Subsets of Grassmannians preserved by mean curvature flows},
    journal={Comm. Anal. Geom.},
    volume={13},
    date={2005},
    number={5},
    pages={981--998},
     review={\MR{2216149 (2006m:53102)}},
 }	
 
 \bib{wa6}{article}{
  author={Wang, Mu-Tao},
    title={Remarks on a class of solutions to the minimal surface system},
     BOOKTITLE ={Geometric evolution equations},
    journal={Contemp. Math.},
    volume={367},
    date={2005},
    pages={229--235},
     review={\MR{2215762 (2005i:53010)}},
 }	
 


\end{biblist}
\end{bibdiv}
\end{document}